\documentclass{amsart}
\usepackage[pdftex]{graphicx} 
\usepackage{color}
\usepackage[T1]{fontenc}
\usepackage{lmodern}
\usepackage[english]{babel} 
\usepackage[utf8]{inputenc}

\usepackage[centering,margin=2.5cm]{geometry}
\usepackage{upgreek}

\usepackage{amsfonts}
\usepackage{verbatim}
\usepackage{amsmath}
\usepackage[all, cmtip]{xy}\usepackage{amssymb}
\usepackage{amsthm}
\usepackage{extpfeil}
\usepackage{bbm}
\usepackage{paralist}
\usepackage[T1]{fontenc}
\usepackage[colorlinks,citecolor=magenta,urlcolor=magenta,pdftex]{hyperref}

\usepackage{pgf}
\usepackage{tikz}
\usetikzlibrary{arrows,automata}
\usetikzlibrary{positioning}

\usepackage{amscd}

\usepackage{nicefrac}
\usepackage{microtype}

\usepackage{mathrsfs}
\usepackage[font=small,labelfont=bf]{caption}
\usepackage[labelformat=simple]{subcaption}

\makeatletter
\newtheorem*{rep@theorem}{\rep@title}
\newcommand{\newreptheorem}[2]{%
\newenvironment{rep#1}[1]{%
 \def\rep@title{#2~\ref{##1}}%
 \begin{rep@theorem}}%
 {\end{rep@theorem}}}
 
\makeatother

\theoremstyle{plain}
\newtheorem*{thm*}{Theorem}
\newtheorem{thm}{Theorem}

\newreptheorem{mthm}{Theorem}
\newreptheorem{mcor}{Corollary}
\newreptheorem{mlem}{Lemma}
\newreptheorem{thm}{Theorem}

\newtheorem{lem}[thm]{Lemma}
\newtheorem*{lem*}{Lemma}

\newtheorem{conj}[thm]{Conjecture}

\theoremstyle{definition}

\newcommand{\RS}{\mathscr{R}}
\newcommand{\se}{\ \approx\ }
\newcommand{\ep}{\varepsilon}

 \newcommand{\Z}{\mathbb{Z}}
 \newcommand{\R}{\mathbb{R}}

\newcommand{\vanish}[1]{}

\def\({\left(}
\def\){\right)}
\def\no={\,{\,|\!\!\!\!\!=\,\,}}

\def\rank{\text\rm{rank}}

\def\no={\,{\,|\!\!\!\!\!=\,\,}}

\def\conv{\text{\rm{conv}}}

\newcommand{\xqedhere}[2]{%
  \rlap{\hbox to#1{\hfil\llap{\ensuremath{#2}}}}}

\newcommand\Defn[1]{\emph{#1}}
\newcommand{\cm}[1]{}
\newcommand\mc[1]{\mathcal{#1}}

\newcommand\mbf[1]{\mathbf{#1}}

\newcommand\mr[1]{\mathrm{#1}}

\newcommand{\bigslant}[2]{{\raisebox{.3em}{$#1$} \Big/ \raisebox{-.3em}{$#2$}}}
\renewcommand\emptyset{\varnothing}

\newcommand\RR{\mathbb{R}}

\title[The Universality theorem for neighborly polytopes]{The universality theorem for neighborly polytopes}
\author{Karim A.~Adiprasito}
\author{Arnau Padrol}
\address{Institut des Hautes \'Etudes Scientifiques, Bures-sur-Yvette, France}
\email{adiprasito@math.fu-berlin.de, adiprasito@ihes.fr}
\address{Institut f\"ur Mathematik, %
Freie Universit\"at Berlin, %
Germany}
\email{arnau.padrol@fu-berlin.de}

\keywords{Realization space, universality theorem, simplicial polytope, neighborly polytope}
\subjclass[2010]{Primary 52B40; Secondary 52C40, 14P10}
\date{\today}
\thanks{K.~A.~Adiprasito acknowledges support by an EPDI postdoctoral fellowship and by the Romanian NASR,
CNCS --- UEFISCDI, project PN-II-ID-PCE-2011-3-0533. The research of A.~Padrol is supported by the DFG
Collaborative Research Center SFB/TR~109 ``Discretization
in Geometry and Dynamics''.}

\begin{document}
\begin{abstract}
In this note, we prove that every open primary basic semialgebraic set is stably equivalent to the realization space of a neighborly simplicial polytope. This in particular
provides the final step for Mn\"ev's proof of the universality theorem for simplicial polytopes.
\end{abstract}

\maketitle

\section{Introduction}

Mn\"ev's Universality Theorem was a fundamental breakthrough in the theory of 
oriented matroids and convex polytopes. It states that the realization spaces 
of oriented matroids and polytopes, i.e.\ the spaces of point configurations 
with fixed oriented matroid/face lattice, can be arbitrarily complex. It comes in four
flavours:

\begin{thm}[Universality Theorem \cite{MnevRoklin}]\label{thm:universality}
Let $V$ be a primary basic semialgebraic set defined over $\Z$, then 
 \begin{compactenum}[\bf(i)]
  \item\label{it:UOM}  there is an \textbf{oriented matroid} of rank $3$ whose realization space is stably equivalent to $V$, and
  \item\label{it:UP}  there is a \textbf{polytope} whose realization space is stably equivalent to $V$;
 \end{compactenum}
if moreover $V$ is open, then 
 \begin{compactenum}[\bf(i)]
  \setcounter{enumi}{+2}
 \item\label{it:UUOM}  there is a \textbf{uniform oriented matroid} of rank $3$ whose realization space is stably equivalent to $V$, and
  \item\label{it:USP}  there is a \textbf{simplicial polytope} whose realization space is stably equivalent to $V$.
 \end{compactenum}
\end{thm}

Mn\"ev announced this theorem in 1985 \cite{MnevA} and published a sketch of the proof in 1988 \cite{MnevRoklin}. A more detailed reasoning can be found in his thesis \cite{MnevT} (in Russian). 
Shor \cite{Shor} simplified a key step in Mn\"ev's line of reasoning for part \eqref{it:UOM} and \eqref{it:UUOM}. 

Moreover, part \eqref{it:UOM} of Theorem~\ref{thm:universality} was later elaborated upon by Richter-Gebert \cite{RG95} and G\"unzel \cite{Gunzel}, who proved the stronger \emph{Universal Partition Theorem} for oriented matroids. Using {Lawrence extensions} to {rigidify} the face lattices, it is easy to prove part \eqref{it:UP} from part \eqref{it:UOM} \cite{MnevRoklin}\cite{RG99}. 
Here, a face lattice is called \Defn{rigid} if it uniquely determines the oriented matroid defined by its vertices.
Additionally, Theorem~\ref{thm:universality}\eqref{it:UP} was generalized greatly by Richter-Gebert, who proved that already $4$-dimensional polytopes are universal \cite{RGZ}\cite{RG}.

For a proof of part \eqref{it:USP}, in contrast, only Mn\"ev’s original papers are
available, apart of some preliminary results of Sturmfels \cite{sturmfels1988simplicial} and Bokowski--Guedes de Oliveira \cite{BG}. Moreover,
Mn\"ev’s elaborations for this case in \cite{MnevT}\cite{MnevRoklin} are specially concise and, in our opinion,
incomplete. Hence, we think part~\eqref{it:USP} of Theorem~\ref{thm:universality}, although widely believed to be true, should be considered open until now. It is important to stress that, despite the wrong common belief, Lawrence extensions \emph{cannot} be used to deduce the universality theorem for simplicial polytopes. We use a different approach to rigidify matroids, namely one based on a result of Shemer proving rigidity of neighborly polytopes~\cite{Shemer}.

In particular, we establish here that \emph{neighborly} polytopes, i.e.\ $d$-polytopes with a complete $\lfloor\frac{d}{2}\rfloor$-skeleton, are universal. Even more, we obtain this result with even-dimensional polytopes.
Since all neighborly polytopes of even dimension are simplicial, this provides a proof of Theorem~\ref{thm:universality}\eqref{it:USP}.

\begin{thm}\label{thm:unin}
Every open primary basic semialgebraic set defined over $\Z$ is stably equivalent to the realization space of some neighborly $(2n-4)$-dimensional polytope on $2n$ vertices. 
\end{thm}

This is in contrast to the case of cyclic polytopes, who have trivial realization spaces \cite[Example~5.1]{BS86}. This holds more generally for all totally sewn and Gale sewn polytopes \cite{Shemer}\cite{Padrol}, of which cyclic polytopes are a particular case. These are large families of neighborly polytopes. Indeed, the number of Gale sewn polytopes with $n$ vertices in a fixed dimension $d$ is asymptotically at least $n^{\lfloor\frac{d}{2}\rfloor n(1-o(1))}$, which is the current best lower bound for the number of \emph{all} polytopes \cite{Padrol}.

\begin{thm}\label{thm:sewn} 

The (homogenized) realization space of any even-dimensional neighborly $d$-polytope on $n$ vertices constucted with the extended sewing or the Gale sewing construction
is contractible, and in fact an open, piecewise smooth $(d+1)\cdot (n-d-1)$-ball.
\end{thm}

Encouraged by the work of Richter-Gebert, we make the daring conjecture that universality holds even when we restrict to neighborly $4$-dimensional polytopes.

\begin{conj}
Every open primary basic semialgebraic set defined over $\Z$ is stably equivalent to the realization space of some neighborly (and hence simplicial) $4$-polytope. 
\end{conj}

A universality conjecture for simplicial $4$-polytopes is supported by the existence of simplicial $4$-polytopes without the isotopy property, that is, with disconnected realization space \cite{BG}.

The following idea provides further motivation for our conjecture. Altshuler and Steinberg proved that vertex figures of neighborly $4$-polytopes are always $3$-dimensional stacked polytopes \cite{AlthsulerSteinberg}. Despite the fact that realization spaces of stacked polytopes are trivial, their oriented matroids can be complicated: Notice that if $M$ is any planar point configuration, then there exists a stacked $3$-polytope $P$ with a distinguished vertex $v$ such that the contraction of $v$ in $P$ coincides with $M$.

Now, realization spaces of planar point configurations are universal, and it is conceivable that these realization spaces can be captured by constructing neighborly $4$-polytopes having them as edge contractions, combined with the fact that all neighborly 4-polytopes are rigid.

\section{Universality for simplicial neighborly polytopes}

The \Defn{realization space} ${\RS}(M)$ of an oriented matroid $M$ is the set 
of vector configurations that share the same oriented matroid, i.e.\ if $M$ is of rank $d$, 
and $E=E(M)$ is the ground set of $M$, then
\[{\RS}(M)=\bigslant{\{X\in \RR^{E\times d} : \text{ $X$ realizes $M$}\}}{\mr{GL}(\R^d)}\]
Similarly, if $P$ is a polytope in $\R^d$, and $V=V(P)$ is its vertex set, then we define its (homogenized) realization space as
\[{\RS}(P)=\bigslant{\{X\in \RR^{V\times (d+1)} : \text{the positive span of $X$ realizes the pointed cone over $P$}\}}{\mr{GL}(\R^{d+1})}\]
Hence, the realization space ${\RS}(P)$ of a polytope is the union of the realization spaces of all oriented matroids whose (Las Vergnas) face lattice (i.e.\ the dual to the lattice of positive cocircuits) coincides with the face lattice of $P$, see \cite[Sec. 9.5]{BLSWZ}. 

A \Defn{basic semialgebraic set} in $\R^d$ is the set of solutions to a finite number of rational polynomial equalities and inequalities; it is called \Defn{primary} if all the inequalities in its definition are strict.
A basic semialgebraic set $S\subset \RR^d$ is a \Defn{stable projection} of a basic semialgebraic set $T\subset \RR^{d+d'}$ if, for the projection $\pi:\RR^{d+d'}\rightarrow \RR^{d}$, we have that $\pi(T)=S$ and that for every $\mbf x\in S$, the fiber $\pi^{-1}(\mbf x)$ is the relative interior of a non-empty polyhedron defined by equalities and strict inequalities that depend polynomially on $\mbf x$.
Two basic semialgebraic sets $S$ and $T$ are \Defn{rationally equivalent} if there is a homeomorphism $f:S \rightarrow T$ such that $f$ and $f^{-1}$ are rational functions.
Two basic semialgebraic sets $S$ and $T$ are \Defn{stably equivalent} if they belong to the same equivalence class generated by stable projections and rational equivalences. We denote the stable equivalence of $S$ and $T$ as $S\se T$. We refer to \cite{RG}\cite{RG99} for more detailed definitions of these concepts.

A technique that pervades our proofs is the use of \Defn{lexicographic extensions} (see~\cite[Section~7.2]{BLSWZ}). When $M$ is realized by a vector configuration $V$, then its lexicographic extension by $[a_1^{\sigma_1},\cdots, a_k^{\sigma_k}]$, where $a_i$ are elements of $V$ and $\sigma_i$ are signs, is realized by adjoining to $V$ the vector ${\sigma_1}a_1+\ep{\sigma_2}a_2+\cdots+\ep^{k-1} {\sigma_k}a_k$ for any $\ep>0$ small enough (cf. Figure~\ref{fig:lex}). The following lemma is straightforward, see also \cite[Lemma~8.2.1 and Proposition~8.2.2]{BLSWZ}.

\begin{lem}\label{lem:lex}
Let $M$ denote any oriented matroid, and let $M[a_1^{\sigma_1},\dots, a_k^{\sigma_k}]$ denote a lexicographic extension of~$M$. Then the projection induced by deletion
\begin{align*}
\RS(M[a_1^{\sigma_1},\dots, a_k^{\sigma_k}])\ \ &\longrightarrow\ \ \RS(M)\\
X\ \ &\longmapsto\ \ X_{|E(M)}
\end{align*}
 is surjective, and its fibers are (polynomially parametrized) polyhedra of dimension $\rank\{a_i\, :\, i \in [k]\}$.
\end{lem}

\begin{figure}[htbf]
\centering 
 \includegraphics[width=0.4\linewidth]{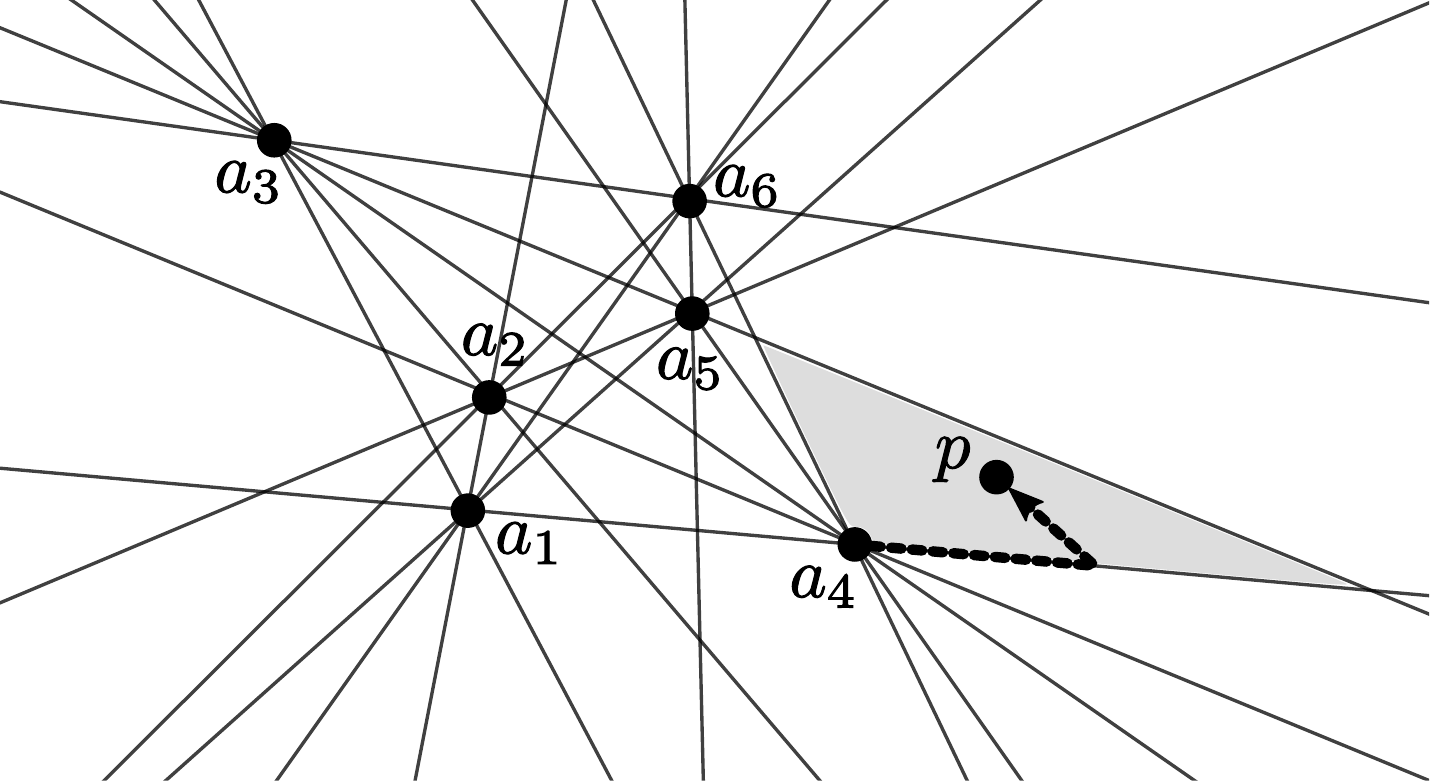} 
\caption{\small A lexicographic extension $M\longrightarrow M[a_4^+,a_1^-,a_6^+]$ and the fiber of the corresponding projection $\RS(M[a_4^+,a_1^-,a_6^+])\longrightarrow\RS(M)$.} 
  \label{fig:lex}
\end{figure}

We start with an observation of Mn\"ev/Shor concerning universality of uniform oriented matroids \cite{MnevA}\cite{Shor}, 
which is proven using {constructible} oriented matroids and a substitution
technique from \cite{MnevRoklin}\cite{JMSW}.

\begin{lem}\label{lem:Shor}
For every open primary basic semialgebraic set $S$ defined over $\Z$, there exists a rank~$3$ uniform oriented matroid $M$
such that
\[S\se \RS(M).\]
\end{lem}

The key step of the proof of Theorem~\ref{thm:unin} is to use a construction of Kortenkamp, who proved that
every $d$-dimensional point configuration of at most $d+4$ points appears as a face
figure of a neighborly polytope~\cite{Kortenkamp}. For larger point configurations, this is
still an open problem, first asked by Perles.

For the proof, recall that in oriented matroid theory a polytope is called \Defn{rigid} if the face lattice of $P$ determines the oriented matroid $M$
spanned by the vertices of $P$ (see \cite[Section 6.6]{Z}). In particular, for all rigid polytopes, we have
$\RS(M)=\RS(P)$.

\begin{lem}\label{lem:K}
For every uniform oriented matroid $M$ of rank $3$ on $n$ elements, there exists a neighborly polytope $P$ with $2n$ vertices in dimension $2n-4$ such that
\[\RS(M)\se \RS(P).\]
\end{lem}

\begin{proof}
By a theorem of Kortenkamp \cite[Theorem 1.2]{Kortenkamp}, every realizable oriented matroid of rank $3$ can be extended to the Gale dual of an even-dimensional neighborly polytope by performing $n$ lexicographic extensions, obtaining a rank $3$ matroid  $\widetilde{M}$ on $2n$ elements. By 
Lemma~\ref{lem:lex} we obtain
\[\RS(M)\se \RS(\widetilde{M}).\]
Let $P$ be the face lattice of the Gale dual of $\widetilde M$. Now, oriented matroid duality preserves realization spaces and, by \cite[Theorem 2.10]{Shemer} and \cite[Theorem 4.2]{Sturmfels88}, every neighborly polytope of even dimension is rigid. Hence,
\[\RS(P)\ \se\ \RS(\widetilde{M}). \qedhere\]
\end{proof}

Together with Lemma \ref{lem:Shor}, this finishes the proof of Theorem~\ref{thm:unin}. 
It remains to characterize the realization spaces of sewn and Gale sewn polytopes.

Let $P$ be a polytope with a flag of faces $\mc F=\{\emptyset=F_0\subset F_1\subset F_2 \subset \dots \subset F_k \subset F_{k+1}=P\}$, and define $U_i=V(F_i)\setminus V(F_{i-1})$.
A point $p$ is said to be \Defn{sewn} onto $P$ through $\mc F$ if it realizes the lexicographic extension of $P$ by $[U_1^+,U_2^-,U_3^+,\dots,U_{k+1}^{(-1)^{k}}]$, where these sets represent their elements in any order.
Shemer~\cite{Shemer} proved that if $P$ is even dimensional and neighborly and $p$ is sewn through a \Defn{universal flag} then $\conv(P\cup p)$ is also neighborly. Here, a \Defn{universal flag} of $P$ is a flag consisting of faces in every odd dimension, and such that the quotients of $P$ by these faces are still neighborly.
This is extended in \cite{Padrol} by relaxing the condition of being a universal flag to containing a universal subflag. 
This technique of generating neighborly polytopes --- iteratively sewing starting from a cyclic polytope --- is called \Defn{the (extended) sewing construction} (see \cite{Shemer} and \cite[Section~3]{Padrol} for details).

Similarly, let $M$ be the oriented matroid of a neighborly $d$-polytope $P$, with dual $M^*$, and let $N$ be an oriented matroid whose dual is obtained by doing first a lexicographic extension in general position of $M^*$ by $p=[a_1^{\sigma_1},\dots, a_r^{\sigma_r}]$ followed by a lexicographic extension by $q=[p^{-},a_1^-,\dots, a_{r-1}^-]$. Then $N$ is also the oriented matroid of a neighborly polytope (of dimension $d+2$). This operation is called \Defn{Gale sewing}, and the neighborly polytopes obtained by repeating this procedure from a polygon or a $3$-polytope are called \Defn{Gale sewn} (cf.~\cite{Padrol}).

Cyclic polytopes arise as a special case of these constructions: they can be obtained by repeatedly extended sewing from a simplex, as well as by Gale sewing from a polygon or certain stacked $3$-polytope \cite{Padrol}.

\begin{proof}[\textbf{Proof of Theorem \ref{thm:sewn}}]
Observe first that, since these are even-dimensional neighborly polytopes, they are rigid and therefore it suffices to argue at the level of realization spaces of oriented matroids (instead of polytopes).
Moreover, up to a change to the dual in Gale sewing that does not affect the realization space, both constructions are based on performing a sequence of lexicographic extensions starting on a configuration with trivial realization space. 

Hence, at the oriented matroid level, their realizations
spaces are open and contractible. Indeed, in each step the fibers of the deletion map are
open polyhedra (Lemma~\ref{lem:lex}). Even more, since these fibers are piecewise polynomially (and hence piecewise smoothly)
parametrized, then the whole realization space is a piecewise smooth ball.
\end{proof}

{\small \textbf{Acknowledgements} We want to thank G\"unter Ziegler for his insightful comments on a previous version of this manuscript.
Also, we wish to thank Nikolai Mn\"ev, J\"urgen Richter-Gebert and Bernd Sturmfels for helpful discussions concerning the history of the universality theorem, and the state of the universality theorem for simplicial polytopes in particular.
Finally, we want to thank Ivan Izmestiev for translating part of N. Mn\"ev's doctoral thesis, and
Hiroyuki Miyata for sparking our interest in this problem.}

{\small
\newcommand{\etalchar}[1]{$^{#1}$}
\providecommand{\noopsort}[1]{}\providecommand{\noopsort}[1]{}
\providecommand{\bysame}{\leavevmode\hbox to3em{\hrulefill}\thinspace}
\providecommand{\MR}{\relax\ifhmode\unskip\space\fi MR }
\providecommand{\MRhref}[2]{%
  \href{http://www.ams.org/mathscinet-getitem?mr=#1}{#2}
}
\providecommand{\href}[2]{#2}

}
\end{document}